\documentclass{amsart}

\usepackage{amssymb,amsmath,amsfonts}
\usepackage{bm}  
\usepackage{amsmath}  
\usepackage{mathtools}
\usepackage{amsthm}
\usepackage{enumitem}
\usepackage{graphicx}
\usepackage[utf8]{inputenc}
\usepackage[T1]{fontenc}
\usepackage{url}

\newtheorem{theorem}{Theorem}[section]
\newtheorem{lemma}[theorem]{Lemma}
\newtheorem{prop}[theorem]{Proposition}

\theoremstyle{definition}

\newtheorem{theoman}{Theorem}[section]
\newtheorem{probman}[theoman]{Problem}

\theoremstyle{remark}

\numberwithin{equation}{section}

\begin{document}

\title{Packing minima of convex bodies}


\author{Mei Han}
\address{Institut f{\"u}r Mathematik, Technische Universit{\"a}t Berlin, Berlin, Germany}
\curraddr{}
\email{ludy\_han@163.com}
\thanks{The research of the first-named author is supported by the Sino-German (CSC-DAAD) Postdoc Scholarship Program (57718047).}

\author{Martin Henk}
\address{Institut f{\"u}r Mathematik, Technische Universit{\"a}t Berlin, Berlin, Germany}
\curraddr{}
\email{henk@math.tu-berlin.de}
\thanks{}

\author{Fei Xue}
\address{Institute of Mathematics, School of Mathematical Sciences, Nanjing Normal University, Nanjing, China}
\curraddr{}
\email{05429@njnu.edu.cn}
\thanks{The research of the third-named author is supported by the National Natural Science Foundation of China (NSFC 12201307).}


\date{\today}

\dedicatory{}

\begin{abstract} In 2021, Henk, Schymura and Xue introduced packing minima, associated with a convex body and a lattice, as packing counterparts to the covering minima of Kannan and Lov{\'a}sz. Motivated by conjectures on the volume inequalities for the successive minima, we generalized the definition of the packing minima to the class of all convex bodies that contain the origin in their interior. For these packing minima, we presented several novel volume inequalities and calculated the specific values of the packing minima for several special convex bodies. 
\end{abstract}

\maketitle

\section{\bf Introduction}\label{section:introduction}
Let $\mathcal{K}^n$ be the set of all convex bodies, i.e., compact convex sets with non-empty interior in the $n$-dimensional Euclidean space $\mathbb{R}^n$,  and let ${\mathcal K}^n_{(o)}$ be the subset of all convex bodies containing the origin in their interior, i.e., ${\bf 0}\in\mathrm{int}(K)$. The family of all $o$-symmetric convex bodies, i.e., $K=-K$, is denoted by $\mathcal{K}_s^n$. If the centroid of a convex body $K$, 
$$
\mathrm{cen}(K)=\frac{1}{\mathrm{vol}(K)}\int_K{\bf x}~\mathrm{d}^n{\bf x}
$$
is the origin, the $K$ is called centered and all centered convex bodies are denoted by $\mathcal{K}_c^n$; here  $\mathrm{vol}(K)$ denotes the volume, i.e., the $n$-dimensional Lebesgue measure of $K$. 

A lattice $\Lambda$ is a discrete subgroup of $\mathbb{R}^n$ 
and by $\det(\Lambda)$ we denote its determinant, i.e., if the dimension of $\Lambda$ is $k$, then $\det(\Lambda)$ is the ($k$-dimensional) volume of a fundamental cell of $\Lambda$.  
 Let $\mathbb{Z}^n\subset \mathbb{R}^n$ be the integer lattice,  and let $\mathcal{L}^n$ be the set of all $n$-dimensional lattices in $\mathbb{R}^n$. 
 
%

For $K\in{\mathcal K}^n_{(o)}$ and $\Lambda\in\mathcal{L}^n$, the {\it $i$-th successive minimum} $\lambda_i(K,\Lambda)$ is the smallest dilation factor $\lambda>0$ such that $\lambda K$ contains at least $i$ linearly independent lattice points of $\Lambda$, i.e.,
\begin{align*}
    \lambda_i(K,\Lambda):=\min\{\lambda>0:\dim(\lambda K\cap\Lambda)=i\}.
\end{align*}
Apparently, $\lambda_1(K,\Lambda)\le\lambda_2(K,\Lambda)\le\cdots\le\lambda_n(K,\Lambda)$. The concept of successive minima was firstly introduced for $o$-symmetric convex bodies by Minkowski in his fundamental and guiding book "Geometrie der Zahlen" in 1896 \cite{minkowski1896}.
In order to make this concept of successive minima  available for arbitrary convex bodies we also need the difference body $K_s$ of a convex body $K\in\mathcal{K}^n$ defined by $K_s:=K-K$. By definition, $K_s$ is an $o$-symmetric convex body.

Minkowski's so called {\it second theorem on successive minima} establishes a beautiful 
upper and lower bound on the volume of an $o$-symmetric convex body in terms of its successive minima (see, e.g., \cite{minkowski1896}) which can easily be stated in a more general form for arbitrary convex bodies (see).

\begin{theoman}[Minkowski's second theorem on successive minima]
    Let $K\in\mathcal{K}^n$ and $\Lambda\in\mathcal{L}^n$. Then
    \begin{align*}
     \frac{1}{n!}\overset{n}{\underset{i=1}{\prod}}\frac{1}{\lambda_i(K_s,\Lambda)}\le\frac{\mathrm{vol}(K)}{\det(\Lambda)}\le\overset{n}{\underset{i=1}{\prod}}\frac{1}{\lambda_i(K_s,\Lambda)}.  
    \end{align*}
\label{thm:minkowski}    
\end{theoman}
Boths bounds are best possible. The lower bound follows by an easy inclusion argument, the upper bound, however, is considered as a deep and important result in Geometry of Numbers. 

Hence, possible generalizations and extensions have been studied  over the last hundered years. One very prominent one is due to Davenport from 1946. In order to present it we need  
\begin{align*}
    \delta^l(K):=\sup\left\{\frac{\mathrm{vol}(\lambda_1(K_s,\Lambda)K)}{\det(\Lambda)}:\Lambda\in\mathcal{L}^n\right\},
\end{align*}
the so-called {\it the lattice packing density} of $K$, i.e., the density of a densest lattice packing of $K$. Note that $\delta^l(K)\leq 1$. The first successive minimum $\lambda_1(K_s,\Lambda)$ is also  called {\it the packing radius} for $K$ and $\Lambda$, as it is the largest dilation factor $\rho>0$ such that 
$$\mathrm{int}(\rho K+{\bf x})\cap\mathrm{int}(\rho K+{\bf y})=\emptyset$$
for any two different lattice points {\bf x} and {\bf y} in $\Lambda$. Here, $\mathrm{int}(\cdot)$ denotes the interior of a set. By definition of the packing density it is
\begin{equation}
    \frac{\mathrm{vol}(K)}{\det(\Lambda)}\le\delta^l(K)\cdot\left(\frac{1}{\lambda_1(K_s,\Lambda)}\right)^n,
\label{eq:davenport1}    
\end{equation}
and  Davenport asked for the following strengthening:  
\begin{probman}[Davenport \cite{davenport1946}]\label{davenport's problem}
    Let $K\in\mathcal{K}^n$. Is it true that
    \begin{align}\label{ineq:Davenport}
       \frac{\mathrm{vol}(K)}{\det(\Lambda)}\le\delta^l(K)\cdot\overset{n}{\underset{i=1}{\prod}}\frac{1}{\lambda_i(K_s,\Lambda)}~? 
    \end{align}
\end{probman}
Obviously, this would be a significant sharpening of Minkowski's upper bound in Theorem \ref{thm:minkowski}.  
So far this question has only been affirmatively answered for $n=2$ and for ellipsoids by Minkowski (see \cite{minkowski1896} and \cite{gruber2007}), the case $n=3$ has been settled by Woods \cite{woods1956}. For more information, we refer to \cite{gruber2007}.

Another direction of extensions is to evaluate the successive minima directly with respect to the convex body, instead with respect to the difference body which a priori only yields optimal results for $o$-symmetric convex bodies.  In \cite{hhh2016}, Henk et al. studied centered convex bodies in this respect and proved among others an optimal lower bound.

\begin{theoman}[\cite{hhh2016}]\label{lvsm2016}
    Let $K\in\mathcal{K}_c^n$ and $\Lambda\in\mathcal{L}^n$. Then,
    \begin{align*}
         \frac{n+1}{n!}\overset{n}{\underset{i=1}{\prod}}\frac{1}{\lambda_i(K,\Lambda)}\le\frac{\mathrm{vol}(K)}{\det(\Lambda)}
    \end{align*}
 with equality  if and only if there exits a basis ${\bf b}_1,...,{\bf b}_n$ of $\Lambda$ and  positive numbers $\nu_1,...,\nu_n>0$ such that
 \begin{align*}
     K=\mathrm{conv}\{\nu_1{\bf b}_1,...,\nu_n{\bf b}_n,-(\nu_1{\bf b}_1+\cdots+\nu_n{\bf b}_n)\}.
 \end{align*}
\end{theoman}
Here, $\mathrm{conv}(\cdot)$ denotes the convex hull of a set. A corresponding upper bound is an open problem and it is related to the so-called Ehrhart conjecture; see \cite{ahenk2023} for information on it.

In 2021, Henk et al. \cite{hsx2021} introduced packing minima  as a counterpart to  the covering minima of Kannan and Lov{\'a}sz \cite{kl1988} and in order to interpolate between the successive minima of $K$ and the ones of the  polar body. We slightly generalize their definition to the class of all convex bodies containing the origin in their interior.  
For  $K\in\mathcal{K}^n_{(o)}$, $\Lambda\in\mathcal{L}^n$ and for $1\leq i\leq n$,  we define the {\it $i$-th packing minimum} $\rho_i(K,\Lambda)$  as
    $$\rho_i(K,\Lambda)=\inf\{\rho>0:(\rho K+L)\cap\Lambda\neq L\cap\Lambda,\text{ for all }L\in\mathcal{G}(\mathbb{R}^n,n-i)\}.$$
    Here $\mathcal{G}(\mathbb{R}^n,k)$ denotes the Grassmannian of all $k$-dimensional linear subspaces of $\mathbb{R}^n$. 
Our first result is an inequality in the spirit of Theorem \ref{lvsm2016} for these packing minima.
\begin{theorem}\label{thm:volume1}
Let $K\in\mathcal{K}_c^n$ and $\Lambda\in\mathcal{L}^n$. Then 
\begin{align}\label{vinequality1a}
    \frac{n+1}{n!}\overset{n}{\underset{i=1}{\prod}}\frac1{\rho_i(K,\Lambda)}\leq \frac{\mathrm{vol}(K)}{\det(\Lambda)},
\end{align}
and the inequality is best possible.
\end{theorem}
As $\rho_i(K,\Lambda)\leq \lambda_{n-i+1}(K,\Lambda)$ (see Proposition \ref{prop:trans}) this inequality improves on Theorem \ref{lvsm2016}. Hence, if we have equality in  Theorem \ref{lvsm2016} then also in the Theorem \ref{thm:volume1}. We remark that for arbitrary convex bodies and with respect to the packing minima $\rho_i(K_s,\Lambda)$ of the difference body it was shown in \cite[Theorem 1.2]{hsx2021} that 
\begin{align}\label{vinequality1b}
    \frac{1}{n!}\overset{n}{\underset{i=1}{\prod}}\frac1{\rho_i(K_s,\Lambda)}\leq \frac{\mathrm{vol}(K)}{\det(\Lambda)}. 
\end{align}
Here we complete this lower bound with a "Davenport type"  upper bound. 

\begin{theorem}\label{thm:volume2}
Let $K\in\mathcal{K}^n$ and $\Lambda\in\mathcal{L}^n$. Then,
$$\frac{\mathrm{vol}(K)}{\det(\Lambda)}\le\delta^l(K)\cdot\overset{n}{\underset{i=1}{\prod}}\frac1{\rho_i(K_s,\Lambda)}.$$ 
\end{theorem}

Apparently, this inequality is weaker than the conjectured one in Problem \ref{davenport's problem} (see Proposition \ref{prop:trans}) but it taks into account the shape of the body via its  lattice density and, in particular, it improves on the upper bound \cite[Theorem 1.2]{hsx2021}. 

In general, packing  minima are hard to compute. For a convex body $K$ (up to an invertible diagonal transformation) between the unit cross-polytope and the unit cube, we always have $\rho_i(K_s,\mathbb{Z}^n)=\lambda_{n-i+1}(K_s,\mathbb{Z}^n)$ (see Proposition \ref{prop:cube and crosspolytope}), which implies that this kind of convex body and the integer lattice align with Problem \ref{davenport's problem}, i.e., Davenport's Problem. However for the centered simplex 
\begin{align*}
    T_n:=-\overset{n}{\underset{i=1}{\sum}}{\bf e}_i+(n+1)\mathrm{conv}\{{\bf 0},{\bf e}_1,...,{\bf e}_n\},
\end{align*} 
where ${\bf e}_i$ denotes the $i$-th unit vector, it is a challenging and interesting problem to determine its packing minima.

\begin{theorem}\label{thm:rhoiTnZn} For $i=1,\dots,n$ it holds 
    $${\rho}_i(T_n,\mathbb{Z}^n)=\frac{1}{\left\lfloor\frac{n}{i+1}\right\rfloor+1}.$$ 
\end{theorem}
Observe, that apparently we have  $\lambda_i(T_n,\mathbb{Z}^n)=1$, $i=1,\dots,n$.

The paper is organized as follows. In the next section we will collect some basic properties of the packing minima, in particular, pointing out their interpolating property between the successive minima and those of the polar body. The proofs of Theorem \ref{thm:volume1} and Theorem \ref{thm:volume2} are given in Section \ref{section:volume_inequality}, and the somehow technical proof of Theorem \ref{thm:rhoiTnZn} is presented in Section \ref{section:TnZn}.

\section{\bf Basic properties of packing minima}\label{section:transference_bound}

The packing minima of a convex body which were introduced and studied in \cite{hsx2021} 
are in our notation the functionals $\rho_i(K_s,\Lambda)$. As the symmetry of the body $K_s$ is not essential 
for many of their basic properties we will point out in the proofs only the necessary changes to the proofs presented in \cite{hsx2021}.

But first we introduce some more notation. For $L\in\mathcal{G}(\mathbb{R}^n,k)$ we denote by $L^\perp\in \mathcal{G}(\mathbb{R}^n,n-k)$ its orthogonal complement, and for $S\subset\mathbb{R}^n$ the orthogonal projection of $S$ onto $L$ is denoted by $S|L$. For a lattice $\Lambda\in\mathcal{L}^n$ let 
$$\mathcal{G}(\Lambda,k):=\{L\in\mathcal{G}(\mathbb{R}^n,k):\dim(L\cap\Lambda)=k\}$$
the set of all $k$-dimensional linear {\it lattice planes} of $\Lambda$.  
These are exactly these planes such $\Lambda|L^\perp$ is again a lattice, 
i.e., a discrete subgroup of $\mathbb{R}^n$. 

With these notation we get as in \cite[Lemma 2.1]{hsx2021} 
the following more convenient description of the packing minima.  

\begin{prop} Let ${\mathcal K}^n_{(o)}$ and $\Lambda\in{\mathcal L}^n$. For $1\leq i\leq n$ it holds 
\begin{equation*}
  \rho_i(K,\Lambda)=\max\{\lambda_1(K|L^\perp,\Lambda|L^\perp): L\in \mathcal{G}(\Lambda,n-i)\}.
\end{equation*}
\label{prop:projdef}
\end{prop}
\begin{proof} As in the proof of \cite[Lemma 2.1]{hsx2021} it follows by the characterization of the closed additive subgroups of $\mathbb{R}^n$ that it suffices to consider lattice planes and that  
\begin{equation*}
  \rho_i(K,\Lambda)=  \sup\{\lambda_1(K|L^\perp,\Lambda|L^\perp): L\in \mathcal{G}(\Lambda,n-i)\}.
\end{equation*}
Let $B^n$ be the $n$-dimensional unit ball, and let $r\,B^n\subseteq K$ for some $r>0$. Then for any $L\in \mathcal{G}(\Lambda,n-i)$ we have along with Minkowski's upper bound in Theorem \ref{thm:minkowski} 
$$
\lambda_1(K|L^\perp,\Lambda|L^\perp)\leq \lambda_1(r\,B^n|L^\perp,\Lambda|L^\perp)\leq 2\cdot\left(\frac{\det(\Lambda|L^\perp)}{\mathrm{vol}_i(r\,B^n|L^\perp)}\right)^\frac{1}{i},
$$
where $\mathrm{vol}_i(\cdot)$ denotes the $i$-dimensional volume, i.e., the Lebesgue measure with respect to $L^\perp$. As in the proof of \cite[Lemma 2.1]{hsx2021} it follows from this inequality that only finitely many lattice planes 
have to be considered in order to find $\sup\{\lambda_1(K|L^\perp,\Lambda|L^\perp): L\in \mathcal{G}(\Lambda,n-i)\}$. 
\end{proof}

Next for a set $S\subseteq\mathbb{R}^n$ let $S^\star:=\{{\bf y}\in\mathbb{R}^n: \langle {\bf x},{\bf y}\rangle\leq 1 \text{ for all }{\bf x}\in S\}$ be its {\it polar set}, and for $\Lambda\in \mathcal{L}^n$ let $\Lambda^\star$ be its {\it polar lattice}, i.e., 
$$
\Lambda^\star:= \{{\bf y}\in\mathbb{R}^n :\langle {\bf x},{\bf y}\rangle\in\mathbb{Z}\text{ for all }{\bf x}\in \Lambda\}.
$$
We remark that (see) 
\begin{equation}
          \mathcal{G}(\Lambda^\star,k)=\{L^\perp: L\in \mathcal{G}(\Lambda,n-k)\}.
\label{eq:polarplane}          
\end{equation}

As mentioned in the introduction, the packing minima were introduced in order to interpolate between the successive minima and the ones of the polar body (with respect to the polar lattice). In our setting it reads as follows. 

\begin{prop} Let ${\mathcal K}^n_{(o)}$, $\Lambda\in{\mathcal L}^n$ and let $\gamma\in(0,1]$ such that $-\gamma K\subseteq K$. For $1\leq i\leq n$ it holds 
$$
                       \frac{\gamma}{\lambda_i(K^\star,\Lambda^\star)} \leq \rho_i(K,\Lambda)\leq\lambda_{n-i+1}(K,\Lambda).
$$
Moreover, we have  $\rho_n(K,\Lambda)=\lambda_1(K,\Lambda)$ and $\rho_1(K,\Lambda)\leq \frac{1}{\lambda_1(K^\star,\Lambda^\star)}$.  
\label{prop:trans}
\end{prop}

Before giving the proof we remark that for an $o$-symmetric convex body $K$ we may choose $\gamma=1$ 
and hence, replacing $K$ by $K-K$ in the above statement gives \cite[Proposition 1.1]{hsx2021}. 

For the lower bound, assume that $n\ge2$, and we would say that $\gamma=\frac{1}{n}$ is a "nearly best" parameter. Note the centered simplex $T_n=-\overset{n}{\underset{i=1}{\sum}}{\bf e}_i+(n+1)\mathrm{conv}\{{\bf 0},{\bf e}_1,...,{\bf e}_n\}$, mentioned in Section \ref{section:introduction}, and apparently, $-\frac{1}{n}T_n\subseteq T_n$. Theorem \ref{thm:rhoiTnZn} yields that $$\rho_i(T_n,\mathbb{Z}^n)\ge\frac{2}{n+2}\ge\frac{1}{n}$$ fro $i=1,2,...,n$.

Turn to the simplex
$$T_n^\star=\mathrm{conv}\left\{-{\bf e}_1, -{\bf e}_2, ..., -{\bf e}_n, \sum_{i=1}^n {\bf e}_i\right\}.$$
By the remark after Theorem \ref{thm:volume1}, Theorem \ref{lvsm2016} and Proposition \ref{prop:trans} we have for $1\leq i\leq n$ 
\begin{equation*}\label{eq:Tn*}
     \rho_i(T_n^\star,\mathbb{Z}^n)=\lambda_{n-i+1}(T_n^\star,\mathbb{Z}^n)=1, 
\end{equation*}
which also implies a similar lower bound for $\rho_i(T_n,\mathbb{Z}^n)$ by Proposition \ref{prop:trans}. 
Then we can see that $\gamma=\frac{1}{n}$ is definitely a "nearly best" parameter and $T_n$ is a "nearly best" candidate associated with this parameter.

\begin{proof} The upper bound can be proven in the same way as in \cite[Proposition 1.1]{hsx2021}, and by definition of $\rho_n(K,\Lambda)$ it equals to $\lambda_1(K,\Lambda)$. 

For the lower bound, we fix an $L \in\mathcal {G}(\Lambda^\star,i)$ such that 
$\lambda_i(K^\star\cap L,\Lambda^\star\cap L)=\lambda_i(K^\star,\Lambda^\star)$. By Proposition \ref{prop:projdef} and \eqref{eq:polarplane} it suffices to show 
$$ 
\gamma\leq \lambda_i(K^\star\cap L,\Lambda^\star\cap L)\lambda_1(K|L,\Lambda|L).
$$
As $K^\star\cap L^\perp=(K|L)^\star$ and $\Lambda^\star\cap L^\perp=(\Lambda|L)^\star$ this can be rewritten as 
$$ 
\gamma\leq \lambda_i((K|L)^\star,(\Lambda|L)^\star)\lambda_1(K|L,\Lambda|L).
$$
In other words, given an $i$-dimensional convex body $M=K|L$ with $-\gamma M\subseteq M$ and an $i$-dimensional lattice $\Gamma=\Lambda|L$ we want to show 
\begin{equation}
\gamma\leq \lambda_i(M^\star,\Gamma^\star)\lambda_1(M,\Gamma).
\label{eq:toshow} 
\end{equation}
This follows, however, as in the symmetric setting (see, e.g., \cite[Theorem 23.2]{gruber2007}): Let ${\bf v}_1\in\Gamma\setminus\{{\bf 0}\}$, and ${\bf w}_1,\dots,{\bf w}_i\in\Gamma^\star$ be linearly independent such that $(1/\lambda_1(M,\Gamma))\, {\bf v}_1\in M$ and $(1/\lambda_i(M^\star,\Gamma^\star))\,{\bf w}_i\in M^\star$.  Then we also have $(1/\lambda_1(M,\Gamma))\, \pm \gamma {\bf v}_1\in M$ and by the definition of the polar body we conclude for $1\leq i\leq n$
$$
\gamma\langle \pm {\bf v}_1,{\bf w}_i\rangle \leq \lambda_i(M^\star,\Gamma^\star)\lambda_1(M,\Gamma).                      $$
As ${\bf w}_1,\dots,{\bf w}_i\in L$ are linearly independent and ${\bf v}_1\ne 0$ one of the inner products must be non-zero. By adjusting the sign of ${\bf v}_1$ we may assume that it is positive, and by definition of the polar lattice we get \eqref{eq:toshow}. 

Finally we consider the case $i=1$. In view of \eqref{eq:polarplane} we have 
\begin{equation*}
 \rho_1(K,\Lambda) =\max\{\lambda_1(K|{\bf u}^\star, \Lambda|{\bf u}^\star): {\bf u}^\star\in \Lambda^\star\setminus\{\bf 0\}\}, 
\end{equation*}
where we write $\cdot|{\bf u}^\star$ in order to denote the orthogonal projection onto the linear hull of the vector ${\bf u}^\star$.  If ${\bf u}^\star\in \Lambda^\star\setminus\{\bf 0\}$ is a primitive vector, i.e.,  ${\bf u}^\star$ is the only non-zero lattice vector in $\Lambda^\star\cap{\rm conv}\{{\bf 0},{\bf u}^\star\}$ then $\det(\Lambda|{\bf u}^*)=1/\Vert{\bf u}^\star\Vert$, where $\Vert\cdot\Vert$ denotes the Euclidean norm. Hence, 
\begin{equation*}
    \lambda_1(K|{\bf u}^\star)= \min\left\{\frac{1}{\mathrm{h}_K({\bf u}^\star)},\frac{1}{\mathrm{h}_K(-{\bf u}^\star)}\right\},
\end{equation*}
where $\mathrm{h}_K:\mathbb{R}^n\to \mathbb{R}$ with $\mathrm{h}_K({\bf u}) =\max\{\langle{\bf x}, {\bf u}\rangle: {\bf x}\in K\}$ is the support function of $K$. Denoting by $\mathrm{r}_{K^\star}:\mathbb{R}^n\setminus\{{\bf 0}\}\to \mathbb{R}$ given by ${\bf u}\mapsto \max\{\mu>0: \mu\,{\bf u} \in K^\star\}$ the radial function of $K^\star$ we have by polarity that $\mathrm{h}_K({\bf u})=1/\mathrm{r}_{K^\star}({\bf u})$. As $\pm {\bf u}^\star\in (1/\mathrm{r}_{K^\star}(\pm {\bf u}^\star))K^\star$ we get 
\begin{equation*}
    \lambda_1(K|{\bf u}^\star )= \min\left\{\mathrm{r}_{K^\star}({\bf u}^\star),\mathrm{r}_{K^\star}(-{\bf u}^\star)\right\} \leq \frac{1}{\lambda_1(K^\star,\Lambda^\star)}.
    \end{equation*}
Hence, $\rho_1(K,\Lambda)\leq 1/\lambda_1(K^\star,\Lambda^\star)$.    
\end{proof}

\section{\bf Volume inequalities for packing minima}\label{section:volume_inequality}

The main purpose of this section is to prove Theorem \ref{thm:volume1}. To this end we will need the following result from \cite{hhh2016}.

\begin{lemma}[\protect{\cite[Lemma 2.2]{hhh2016}}]\label{hhh2016}
Let $K\in\mathcal{K}_c^n$ and let ${\bf u}_1$, $\cdots$, ${\bf u}_n\in K$ be linearly independent. Then
$$\mathrm{vol}(K)\ge\frac{n+1}{n!}|\det({\bf u}_1,\cdots,{\bf u}_n)|.$$ Equality is attained if and only if $K=\mathrm{conv}\{{\bf u}_1,\cdots,{\bf u}_n,-({\bf u}_1+\cdots+{\bf u}_n)\}$. 
\end{lemma}

The proof of Theorem \ref{thm:volume1} essentially follows from the following statement.
\begin{lemma} Let $K\in {\mathcal K}^n_{(o)}$ and $\Lambda\in{\mathcal L}^n$. There exist linearly independent vectors ${\bf v}_1,\dots,{\bf v}_n\in K$ and ${\bf w}_1,\dots,{\bf w}_n\in\Lambda$ such that for $i=1,\dots,n$  
\begin{equation*}
     \rho_{n-i+1}(K,\Lambda)\left({\bf v}_{i}|
     \{{\bf w}_1,\dots,{\bf w}_{i-1}\}^\perp\right) = {\bf w}_i|\{{\bf w}_1,\dots,{\bf w}_{i-1}\}^\perp.
\end{equation*}
\label{lem:vectors}
\end{lemma} 

\begin{proof} We use induction on $n$, and for abbreviation we write $\rho_i=\rho_{i}(K,\Lambda)$. For $n=1$ we only have $\rho_1=\lambda_1(K,\Lambda)$ and so there exists a ${\bf v}_1\in K$ such that 
$\rho_1{\bf v}_1={\bf w}_1\in \Lambda$. Next we assume $n>1$. In view of Proposition \ref{prop:trans} we have $\rho_n=\lambda_1(K,\Lambda)$ and so let   ${\bf v}_1\in K$ such that 
${\bf w}_1=\rho_n{\bf v}_1\in \Lambda$. Next consider the $(n-1)$-dimensional convex body $K'=K|\{{\bf w}_1\}^\perp$, the $(n-1)$-dimensional lattice $\Lambda'=\Lambda|\{{\bf w}_1\}^\perp$ and for $1\leq i\leq n-1$ let $\rho_i'=\rho_i(K',\Lambda')$. According to our inductive process we can find linearly independent vectors ${\bf v}_2',\dots, {\bf v}_{n}'\in K$ and  ${\bf w}'_2,\dots,{\bf w}'_{n}\in\Lambda'$ such that for $i=2,\dots,n$
\begin{equation*}
     \rho_{n-i+1}'\left({\bf v}_{i}'|\{{\bf w}_2,\dots,{\bf w}_{i-1}'\}^\perp\right) ={\bf w}_i'|\{{\bf w}_2',\dots,{\bf w}_{i-1}'\}^\perp.
\end{equation*}
Observe that  for $1\leq i\leq n-1$ we have 
\begin{equation*}
     \rho_i'\leq \rho_i,
\end{equation*} 
and hence with $\overline{\bf v}_i= (\rho_i'/\rho_i){\bf v}_i'\in K'$, $2\leq i\leq n$,  we get  
\begin{equation*}
    \rho_{n-i+1} \left(\overline{\bf v}_i|\{{\bf w}_2',\dots,{\bf w}_{i-1}'\}^\perp\right)  ={\bf w}_i'|\{{\bf w}_2',\dots,{\bf w}_{i-1}'\}^\perp.
\end{equation*}
Finally, let $\alpha_i,\beta_i\in\mathbb{R}$ such that ${\bf v}_{i}=\overline{\bf v}_i+\alpha_i{\bf w}_1\in K$ and ${\bf w}_i={\bf w}_i'+\beta_i{\bf w}_1\in\Lambda$ for $2\leq i\leq n$. Then for $2\leq i\leq n$
\begin{equation*}
\begin{split}
      \rho_{n-i+1}\left({\bf v}_{i}|\{{\bf w}_1,\dots,{\bf w}_{i-1}\}^\perp\right) &= \rho_{n-i+1}\left({\bf v}_{i}|\{{\bf w}_1\}^\perp |\{{\bf w}_2|\{{\bf w}_1\}^\perp,\dots,{\bf w}_{i-1}|\{{\bf w}_1\}\}^\perp\right)\\ 
      &= \rho_{n-i+1}'\left(\overline{\bf v}_{i}|\{{\bf w}'_2,\dots,{\bf w}_{i-1}'\}^\perp\right)\\ 
      &={\bf w}_i'|\{{\bf w}_2',\dots,{\bf w}_{i-1}'\}^\perp \\
      &={\bf w}_i|\{{\bf w}_1,\dots,{\bf w}_{i-1}\}^\perp, 
      \end{split} 
\end{equation*}
and so we have found vectors with the desired property.  
\end{proof}

The proof of Theorem \ref{thm:volume1} is now an immediate consequence of the last two lemmas.

\begin{proof}[{\it Proof of Theorem \ref{thm:volume1}}]
Let the vectors ${\bf v}_1,\dots,{\bf v}_n\in K$ and ${\bf w}_1,\dots,{\bf w}_n\in\Lambda$ as in Lemma \ref{lem:vectors}, and let $\rho_i=\rho_i(K,\Lambda)$, $1\leq i\leq n$.  Then Lemma \ref{hhh2016} gives  

\begin{equation*}
\begin{split}
      \mathrm{vol}(K) & \ge\frac{n+1}{n!}|\det({\bf v}_1,\cdots,{\bf v}_n)|\\
                  & =\frac{n+1}{n!}\frac{1}{\rho_n}|\det({\bf w}_1, {\bf v}_2|\{{\bf w}_1\}^\perp,\dots, {\bf v}_n|\{{\bf w}_1\}^\perp)    \\ & 
                  =\frac{n+1}{n!}\frac{1}{\rho_n}\frac{1}{\rho_{n-1}}|\det({\bf w}_1, {\bf w}_2|\{{\bf w}_1\}^\perp,{\bf v}_3|\{{\bf w}_1\}^\perp\dots, {\bf v}_n|\{{\bf w}_1\}^\perp)\\
                  &=\frac{n+1}{n!}\frac{1}{\rho_n}\frac{1}{\rho_{n-1}}|\det({\bf w}_1, {\bf w}_2,{\bf v}_3|\{{\bf w}_1\}^\perp, \dots, {\bf v}_n|\{{\bf w}_1\}^\perp)\\ 
                  &= \frac{n+1}{n!}\frac{1}{\rho_n}\frac{1}{\rho_{n-1}}|\det({\bf w}_1, {\bf w}_2,{\bf v}_3|\{{\bf w}_1, {\bf w}_2\}^\perp\dots, {\bf v}_n|\{{\bf w}_1, {\bf w}_2\}^\perp)  \\ 
                  &=\cdots \\ 
                  &=\frac{n+1}{n!}\prod_{i=1}^n \frac{1}{\rho_i}\,|\det({\bf w}_1,\cdots,{\bf w}_n)|
                  \\
                  &\geq \frac{n+1}{n!}\prod_{i=1}^n \frac{1}{\rho_i}\,\det(\Lambda).
\end{split}
\end{equation*} 

\end{proof}

Now we turn to the Davenport type upper bound, for which we need the following result from \cite{hsx2021}.

\begin{lemma}[\protect{\cite[Lemma 4.4]{hsx2021}}]\label{lem:rho_projection2}
Let $K\in\mathcal{K}^n$, and let $S=\{j\in\{1,...,n-1\}:\rho_j(K_s,\Lambda)>\rho_n(K_s,\Lambda)\}\not=\emptyset$. Let $k^*=\max\{i: i\in S\}$ be the maximal element in $S$. Then there exists an $L\in\mathcal{G}(\Lambda,n-k^*)$ such that for all $j\in S$,
$$\rho_j(K_s,\Lambda)=\rho_j(K_s|L^\perp,\Lambda|L^\perp).$$ 
\end{lemma}

Theorem \ref{thm:volume2} is a particular case of the next theorem which is a strengthening of \cite[Theorem 4.4]{hsx2021}.

\begin{theorem}\label{volume2a}
Let $K\in\mathcal{K}^n$ and $\Lambda\in\mathcal{L}^n$. Further, let $0=j_0<j_1<\cdots<j_m<j_{m+1}=n$ with $m\in\mathbb{Z}_{\ge0}$ be such that
\begin{equation*}
\begin{split} 
    \mathrm{(a)}&\quad \rho_{j_1}(K_s,\Lambda)>\rho_{j_2}(K_s,\Lambda)>\cdots>\rho_{j_m}(K_s,\Lambda)>\rho_{j_{m+1}}(K_s,\Lambda)=\rho_n(K_s,\Lambda),\\
    \mathrm{(b)}&\quad\rho_j(K_s,\Lambda)\le\rho_{j_i}(K_s,\Lambda),  \quad j\in\{j_{i-1}+1,..., j_i\}, i\in\{1,...,m+1\}.
 \end{split}    
\end{equation*}
Then,
\begin{align}\label{vinequality2}
    \frac{\mathrm{vol}(K)}{\det(\Lambda)}\le\delta^l(K)\cdot\overset{m+1}{\underset{i=1}{\prod}}\frac1{\rho_{j_{i}}(K_s,\Lambda)^{j_{i}-j_{i-1}}}.
\end{align} 
\end{theorem}

\begin{proof} We will use induction on $m$. If $m=0$ then $j_1=n$ and with  $\rho_n(K_s,\Lambda)=\lambda_1(K_s,\Lambda)$ the inequality becomes   \eqref{eq:davenport1}. So let $m\geq 1 $, and  let 
 $$S=\{j_1, j_2,..., j_{m-1}, j_{m}\}.$$ By Lemma \ref{lem:rho_projection2}, there exists an $L\in\mathcal{G}(\Lambda,n-j_m)$ such that for $l\in S$,
$$\rho_l(K_s,\Lambda)=\rho_l(K_s|L^\perp,\Lambda|L^\perp).$$
Then, by assumption we have
\begin{equation*}
    \mathrm{(a_1)}\quad  \rho_{j_1}(K_s|L^\perp,\Lambda|L^\perp)>\rho_{j_2}(K_s|L^\perp,\Lambda|L^\perp)>\cdots>\rho_{j_m}(K_s|L^\perp,\Lambda|L^\perp).
\end{equation*}
Moreover, in view of Proposition \ref{prop:projdef} and assumption (b) from the theorem, we also get for $j\in\{j_{i-1}+1,..., j_i\}$ and $i\in\{1,...,m\}$
\begin{equation*}
  \mathrm{(b_1)}\quad \rho_j(K_s|L^\perp,\Lambda|L^\perp)\le\rho_j(K_s,\Lambda)\leq \rho_{j_i}(K_s,\Lambda) =\rho_{j_i}(K_s|L^\perp,\Lambda|L^\perp).
\end{equation*}
Hence, by our induction approach we conclude 
\begin{equation}
\begin{split}
    \delta^l(K|L^\perp) &\geq \frac{\mathrm{vol}_{j_m}(K|L^\perp)}{\det(\Lambda|L^\perp)}\prod_{i=1}^{m} \rho_{j_i}(K_s|L^\perp,\Lambda|L^\perp)^{j_i-j_{i-1}} \\ &= \frac{\mathrm{vol}_{j_m}(K|L^\perp)}{\det(\Lambda|L^\perp)}\prod_{i=1}^{m} \rho_{j_i}(K_s,\Lambda)^{j_i-j_{i-1}}
\label{eq:induct}    
\end{split}     
\end{equation}
Let $\widetilde\Lambda\subset L^\perp$ be a densest packing lattice of $K|L^\perp$, i.e., 
\begin{equation*} 
   \delta(K|L^\perp)=\frac{\mathrm{vol}_{j_m}(K|L^\perp)}{\det(\widetilde\Lambda)},
\end{equation*} 
and let $\overline\Lambda\subset L$ be $(n-j_m)$-dimensional lattice 
$$
\overline\Lambda=\frac{1}{\lambda_1(K_s,\Lambda)}\left(\Lambda\cap L\right).$$
Then $\widetilde\Lambda\oplus\overline \Lambda$ is a packing lattice of $K$, or equivalently of $\frac{1}{2} K_s$, and with \eqref{eq:induct} we get 
\begin{equation*}
    \begin{split}
        \delta^l(K)&\geq \frac{\mathrm{vol}(K)}{\det(\widetilde\Lambda\oplus\overline \Lambda)} \\ &= \delta(K|L^\perp)\frac{\mathrm{vol}(K)}{\mathrm{vol}_{j_m}(K|L^\perp)} \frac{\lambda_1(K_s,\Lambda)^{n-j_m}}{\det(\Lambda\cap L)} \\   
         &\geq \frac{\mathrm{vol}(K)}{\det(\Lambda|L^\perp)\det(\Lambda\cap L)}\lambda_1(K_s,\Lambda)^{n-j_m}\prod_{i=1}^{m} \rho_{j_i}(K_s,\Lambda)^{j_i-j_{i-1}} \\
         &= \frac{\mathrm{vol}(K)}{\det(\Lambda)}\prod_{i=1}^{m+1} \rho_{j_i}(K_s,\Lambda)^{j_i-j_{i-1}}.
    \end{split}
\end{equation*}
\end{proof}

\begin{proof}[Proof of Theorem \ref{thm:volume2}] With the notation of Theorem \ref{volume2a} we immediately get by \eqref{vinequality2} and condition (b)
\begin{equation*}
    \frac{\mathrm{vol}(K)}{\det(\Lambda)}\le\delta^l(K)\cdot\overset{m+1}{\underset{i=1}{\prod}}\frac1{\rho_{j_{i}}(K_s,\Lambda)^{j_{i}-j_{i-1}}}\leq \delta^l(K)\cdot\overset{n}{\underset{i=1}{\prod}}\frac1{\rho_{{i}}(K_s,\Lambda)}.
\end{equation*} 
\end{proof}

\section{\bf Packing minima of some convex bodies}\label{section:TnZn}

In this section, we will study the packing minima of some particular convex bodies with respect to the integer lattice. 

Firstly, let $C_n=\{{\bf x}=(x_1,x_2,...,x_n)\in\mathbb{R}^n:|x_i|\le1\}$ be the unit cube and $C_n^\star=\rm{conv}\{\pm{\bf e}_i:1\le i\le n\}$ be the unit cross-polytope, the polar body of $C_n$. Furthermore, let $\mathcal{A}_n=\{\rm{diag}({\it a}_1,...,{\it a}_n):{\it a}_i\in\mathbb{R}_{>0}\}$ be the group of all diagonal matrics with positive entries on the diagonal, and let
\begin{equation*}
    \mathcal{K}^n_\mathcal{A}:=\{K\in\mathcal{K}^n:\text{There exists an }A\in\mathcal{A}_n\text{ with }C_n^\star\subseteq AK\subseteq C_n\}.
\end{equation*}
For instance, unconditional bodies or more general, $o$-symmteric locally anti-blocking bodies (see \cite{aass2023,kos2020}) belong to this class.

Let now $K\in\mathcal{K}^n_\mathcal{A}$ with $A=\rm{diag}({\it a}_1,...,{\it a}_n)\in\mathcal{A}_n$. By the definition above we have
\begin{align*}
    \rm{conv}\left\{\pm\frac{1}{{\it a}_i}{\bf e}_i:1\le i\le n\right\}\subseteq K\subseteq\left\{{\bf x}:|x_i|\le\frac{1}{{\it a}_i},1\le i\le n\right\}
\end{align*}
and assuming $a_1\le a_2\le\cdots\le a_n$ we get for $1\le i\le n$
\begin{align*}
    \lambda_i(K,\mathbb{Z}^n)=a_i.
\end{align*}
Apparently, $K^\star\in\mathcal{K}^n_\mathcal{A}$ as well, with associated diagonal matrix $A^{-1}$, and also $K_s=K-K\in\mathcal{K}^n_\mathcal{A}$ with matrix $2A$. Hence we get the following result by Proposition \ref{prop:trans}.

\begin{prop} Let $K\in\mathcal{K}^n_\mathcal{A}$. Then for $1\leq i\leq n$ 
\begin{equation*}
       \frac{1}{\lambda_{i}(K_s^\star,\mathbb{Z}^n)}=\rho_i(K_s,\mathbb{Z}^n)=\lambda_{n-i+1}(K_s,\mathbb{Z}^n).
\end{equation*} 
\label{prop:cube and crosspolytope}
\end{prop}

Next we turn to the simplex 
\begin{align*}
    T_n=-\overset{n}{\underset{i=1}{\sum}}{\bf e}_i+(n+1)\mathrm{conv}\{{\bf 0},{\bf e}_1,...,{\bf e}_n\}. 
\end{align*} 
As Section \ref{section:transference_bound} shows, $\rho_i(T_n^\star,\mathbb{Z}^n)=1$ holds for all $i$, which can be deduced directly by the konwn results. For the completeness, we restate it here as a proposition, which means $ T_n^\star$ and the integer lattice also algin with Problem \ref{davenport's problem}.

\begin{prop}\label{prop:rhoiTn*Zn}
For $i=1,\dots,n$ it holds 
    $${\rho}_i(T_n^\star,\mathbb{Z}^n)=\lambda_{n-i+1}(T_n^\star,\mathbb{Z}^n)=1.$$ 
\end{prop}

Unfortunately, the packing minima of $T_n$ itself are much more involved. 

The proof of  Theorem \ref{thm:rhoiTnZn} 
is based on an interesting and elegant observation: along certain special directions, the projection of the entire simplex $T_n$ coincides with one of its $i$-dimensional subfacets. Moreover, the distribution of the projected integer points happens to satisfy the extremal condition, which allows us to verify the equality by establishing the two opposing inequalities. 

For convenience, let ${\bf v}_0=(-1,-1,...,-1)$, ${\bf v}_1=(n,-1,...,-1)$, $\cdots$, ${\bf v}_n=(-1,...,-1,n)$ be the vertices of $T_n$, i.e., $T_n=\mathrm{conv}\{{\bf v}_0,{\bf v}_1,\cdots,{\bf v}_n\}$, and let $V(T_n)=\{{\bf v}_0,{\bf v}_1,\cdots,{\bf v}_n\}$. Then, we need the following lemmas.

\begin{lemma}\label{lem:upTnZn}
    For $i=1,2,...,n$, let $L\in\mathcal{G}(\mathbb{Z}^n,n-i)$ and let $F$ be an $i$-dimensional face of $T_n$ such that $\dim(F|L^\perp)=i$ and ${\bf 0}\in F|L^\perp$. Then 
    \begin{align*}
    \lambda_1(F|L^\perp,\mathbb{Z}^n|L^\perp)\leq\frac{1}{\left\lfloor\frac{n}{i+1}\right\rfloor+1}.  
    \end{align*}
\end{lemma}

\begin{proof}
As introduced above, $V(T_n)=\{{\bf v}_0,{\bf v}_1,\cdots,{\bf v}_n\}$. Without loss of generality, we assume $F=\mathrm{conv}\{{\bf v}_0,{\bf v}_1,\cdots,{\bf v}_i\}$. As ${\bf 0}\in F|L^\perp$ there exists a $${\bf p}=\tau_0{\bf v}_0+\tau_1{\bf v}_1+\cdots+\tau_i{\bf v}_i\in F\cap L$$ such that $${\bf p}|L^\perp={\bf 0},$$ where $0\le\tau_j\le1$ for $j=0,1,...,i$ and $\tau_0+\tau_1+\ldots+\tau_i=1$. Note that $(F\cap\mathbb{Z}^n)|L^\perp\subseteq (F|L^\perp)\cap(\mathbb{Z}^n|L^\perp)$.

Let $n=(i+1)\cdot m+k$ for $m=\left\lfloor\frac{n}{i+1}\right\rfloor$ and $k\in\{0,1,...,i\}$, then $$n+1=(m+1)+i\cdot m+k.$$ Now we consider the following cases for different values of $k$.
\begin{enumerate}[fullwidth,itemindent=2em,label={\bf Case \arabic*.}] 
    \item For $k=i$, $n+1=(i+1)\cdot(m+1)$, then there is a positive integer $i_0\in\{0,1,...,i\}$ such that $\tau_{i_0}\ge\frac{m+1}{n+1}$, otherwise $\tau_0+\tau_1+\ldots+\tau_i<1$ and ${\bf p}\notin F$, contradictorily. Then 
$${\bf p}+\frac{m+1}{n+1}({\bf v}_j-{\bf v}_{i_0})\in F$$
holds for any $j\in\{0,1,...,i\}\setminus\{i_0\}$. Since $\frac{1}{n+1}({\bf v}_j-{\bf v}_{i_0})|L^\perp\in\mathbb{Z}^n|L^\perp$ and ${\bf p}|L^\perp={\bf 0}$, we have
$$\lambda_1(F|L^\perp,\mathbb{Z}^n|L^\perp)\le\frac{1}{m+1}.$$
  \item For $0\le k\le i-1$, $n+1=(k+1)\cdot(m+1)+(i-k)\cdot m$. 

 If there exists some $i_0$ such that $\frac{m+1}{n+1}\le\tau_{i_0}\le1$, then for any $j\in\{0,1,...,i\}\setminus\{i_0\}$,
  $${\bf p}+\frac{m+1}{n+1}({\bf v}_j-{\bf v}_{i_0})\in F.$$

  If for any $j\in\{0,1,...,i\}$, $0\le\tau_j<\frac{m+1}{n+1}$, we take 
  $${\bf c}=\frac{n_0{\bf v}_0+n_1{\bf v}_1+\ldots+n_i{\bf v}_i}{n+1}\in F\cap\mathbb{Z}^n$$ such that $n_j=m$ or $m+1$ for $j=0,1,...,i$, then 

\begin{align*}
 {\bf p}+(m+1)({\bf c}-{\bf p})=\underset{j=0}{\overset{i}{\sum}}\left[\frac{n_j}{n+1}\cdot(m+1)-m\cdot\tau_j\right]\cdot{\bf v}_j.
\end{align*}
By simple calculation, we have $$\frac{n_j}{n+1}\cdot(m+1)-m\cdot\tau_j\ge0$$
and
$$\underset{j=0}{\overset{i}{\sum}}\left[\frac{n_j}{n+1}\cdot(m+1)-m\cdot\tau_j\right]=1,$$ then
$$ {\bf p}+(m+1)({\bf c}-{\bf p})\in F.$$

By $({\bf c}-{\bf p})|L^\perp={\bf c}|L^\perp\in\mathbb{Z}^n|L^\perp$, we have
$$\lambda_1(F|L^\perp,\mathbb{Z}^n|L^\perp)\le\frac{1}{m+1}.$$
\end{enumerate}

Thus, this proof is finished.

\end{proof}


\begin{lemma}\label{lem:projTnZnFn}
    Let $F$ be an $i$-dimensional face of $T_n$, then there exists an $L'\in\mathcal{G}(\mathbb{Z}^n,n-i)$ such that
    \begin{align*}
       T_n|L'^\perp=F|L'^\perp
    \end{align*}
    and
    \begin{align*}
       (T_n|L'^\perp)\cap(\mathbb{Z}^n|L'^\perp)=(F\cap\mathbb{Z}^n)|L'^\perp. 
    \end{align*}
\end{lemma}

\begin{proof}
    For convenience, we still take the $i$-dimensional face $F$ of $T_n$ with $V(F)=\{{\bf v}_0,{\bf v}_1,\cdots,{\bf v}_i\}$ as an example.

Firstly, we divide $V(T_n)$ into $i+1$ pairwise disjoint subsets $V_0=\{{\bf v}_{01},\cdots,{\bf v}_{0,n_0}\}$, $V_1=\{{\bf v}_{11},\cdots,{\bf v}_{1,n_1}\}$, $\cdots$, $V_i=\{{\bf v}_{i,1},\cdots,{\bf v}_{i,n_i}\}$, where ${\bf v}_{k1}={\bf v}_k$ for $k=0,1,...,i$ and $n_0+n_1+\cdots+n_i=n+1$.

Next, we assume that $L_j$ is the $(n-1)$-dimensional subspace of $\mathbb{R}^n$ parallel to conv$V_j$ and conv$\underset{k=0,k\not=j}{\overset{i}{\cup}}V_k$. Let ${\bf b}_{kl}=\frac{{\bf v}_{k,l+1}-{\bf v}_{k1}}{n+1}$ and we have ${\bf b}_{kl}\in L_j\cap\mathbb{Z}^n$, where $l=1,2,...,n_k-1$ for $k=0,1,...,i$. For convenience, let $$\Phi=\{{\bf b}_{01},\cdots,{\bf b}_{0n_0-1},{\bf b}_{11},\cdots,{\bf b}_{1n_1-1},\cdots,{\bf b}_{i,1},\cdots,{\bf b}_{i,n_i-1}\}.$$
Moreover, let ${\bf b}_k=\frac{{\bf v}_k-{\bf v}_{k+1}}{n+1}$ for $k=0,1,...,i$ and ${\bf v_{i+1}}={\bf v}_0$, then
$$\widetilde{\Phi}_j=\Phi\cup\{{\bf b}_0,{\bf b}_1,\cdots,{\bf b}_{j-2},{\bf b}_{j+1},\cdots,{\bf b}_i\}\subset L_j\cap\mathbb{Z}^n.$$
It is easy to see that $\widetilde{\Phi}_j\cup\{{\bf b}_j\}$ is a group of linearly independent primary vectors in $\mathbb{Z}^n$, implying that they are also a basis of $\mathbb{Z}^n$, i.e.,
$$\mathbb{Z}^n=L_j\cap\mathbb{Z}^n+\{z_j\cdot{\bf b}_j:z_j\in\mathbb{Z}^n\}.$$
Thus, $L_j\in\mathcal{G}(\mathbb{Z}^n,n-1)$.

Assume that $L'=\underset{j=0}{\overset{i}{\cap}}L_j$, then $L'$ is parallel to each $\mathrm{conv}V_j$ and $\Phi$ is a basis of $L'\cap\mathbb{Z}^n$, which yields $L'\in\mathcal{G}(\mathbb{Z}^n,n-i)$. 

For $j_0=0,1,...,i$, we have known that $\widetilde{\Phi}_{j_0}\cup\{{\bf b}_{j_0}\}$ is a basis of $\mathbb{Z}^n$, then it can be deduced that $$\{{\bf b}_k|L'^\perp:k=0,1,...,j_0-2,j_0,j_0+1,...,i\}$$ is a basis of $\mathbb{Z}^n|L'^\perp$. Since $L'$ is parallel to each $\mathrm{conv}V_j$ and for each $k=0,1,...,i$, $V_k=\{{\bf v}_{k1},\cdots,{\bf v}_{k,n_k}\}$, we have $V_k|L'^\perp={\bf v}_k|L'^\perp$, which implies that
$$T_n|L'^\perp=F|L'^\perp=\mathrm{conv}\{{\bf v}_0|L'^\perp, {\bf v}_1|L'^\perp, ..., {\bf v}_i|L'^\perp\}$$
and
$$(T_n|L'^\perp)\cap(\mathbb{Z}^n|L'^\perp)=(F|L'^\perp)\cap(\mathbb{Z}^n|L'^\perp).$$ Note that ${\bf v}_k={\bf v}_{k1}$. Further, since ${\bf b}_k=\frac{{\bf v}_k-{\bf v}_{k+1}}{n+1}$, we can also obtain
$$(\overline{{\bf v}_k{\bf v}_l}\cap\mathbb{Z}^n)|L'^\perp=(\overline{{\bf v}_k{\bf v}_l}|L'^\perp)\cap(\mathbb{Z}^n|L'^\perp),$$
for $k,l=0,1,...,i$ and $k\not=l$,where $\overline{{\bf v}_k{\bf v}_l}=\text{conv}\{{\bf v}_k,{\bf v}_l\}$, then
$$(F\cap\mathbb{Z}^n)|L'^\perp=(F|L'^\perp)\cap(\mathbb{Z}^n|L'^\perp).$$ 
Therefore,  $(T_n|L'^\perp)\cap(\mathbb{Z}^n|L'^\perp)=(F\cap\mathbb{Z}^n)|L'^\perp$.



\end{proof}

\begin{lemma}\label{lem:lowTnZn}
    Let $F$ be an $i$-dimensional face of $T_n$, then there exists an $L'\in\mathcal{G}(\mathbb{Z}^n,n-i)$ such that
    \begin{align*}
       \lambda_1(F|L'^\perp,\mathbb{Z}^n|L^\perp)=\frac{1}{\left\lfloor\frac{n}{i+1}\right\rfloor+1}.
    \end{align*}
\end{lemma}

\begin{proof}
   Take the $(n-i)$-dimensional lattice plane $L'\in\mathcal{G}(\mathbb{Z}^n,n-i)$ constructed in Lemma \ref{lem:projTnZnFn}, and we show that it indeed satisfies the equivalent above.

   Let ${\bf c}=\frac{n_0{\bf v}_0+n_1{\bf v}_1+\ldots+n_i{\bf v}_i}{n+1}\in F$, where $n_j$ is the number of vertices in $V_j$ for $j=0,1,...,i$, as set in the proof of Lemma \ref{lem:projTnZnFn}. Since ${\bf b}_{jl}=\frac{{\bf v}_{j,l+1}-{\bf v}_{j1}}{n+1}\in L'$, where $l=1,2,...,n_j-1$ for $j=0,1,...,i$, we have $${\bf v}_{j1}={\bf v}_{j,l+1}-(n+1){\bf b}_{jl},$$ then $$n_j{\bf v}_{j1}=\underset{l=1}{\overset{n_j}{\sum}}{\bf v}_{j,l}-\underset{l=1}{\overset{n_j-1}{\sum}}(n+1){\bf b}_{jl}.$$
Note ${\bf v}_{j1}={\bf v}_j$ for $j=0,1,2,...,i$. It follows by $T_n\in\mathcal{K}_c^n$ that
\begin{align*}
    \underset{j=0}{\overset{i}{\sum}}\underset{l=1}{\overset{n_j}{\sum}}{\bf v}_{j,l}={\bf 0}
\end{align*}
and
\begin{align*}
 {\bf c} & =\frac{\underset{j=0}{\overset{i}{\sum}}\underset{l=1}{\overset{n_j}{\sum}}{\bf v}_{j,l}-\underset{j=0}{\overset{i}{\sum}}\underset{l=1}{\overset{n_j-1}{\sum}}(n+1){\bf b}_{jl}}{n+1} \\
         &=-\underset{j=0}{\overset{i}{\sum}}\underset{l=1}{\overset{n_j-1}{\sum}}{\bf b}_{jl}\in L'\cap\mathbb{Z}^n,
\end{align*}
then
$${\bf c}|L'^\perp={\bf 0}.$$

Similarly, assume $n+1=(k+1)\cdot(m+1)+(i-k)\cdot m$ for $m=\lfloor\frac{n}{i+1}\rfloor$ and $k\in\{0,1,...,i\}$. In particular, let $n_0=n_1=...=n_k=m+1$ and $n_{k+1}=n_{k+2}=...=n_i=m$, then $${\bf c}=\frac{\overset{k}{\underset{l=0}{\sum}}(m+1){\bf v}_l+\overset{i}{\underset{l=k+1}{\sum}}m{\bf v}_l}{n+1}\in F\cap\mathbb{Z}^n.$$

Take an arbitrary pairwise $i_0,i_1\in\{0,1,...,i\}$. If $\{i_0,i_1\}\subset\{0,1,...,k\}$,
$${\bf c}\pm\frac{m+1}{n+1}({\bf v}_{i_0}-{\bf v}_{i_1})\in\partial(F);$$
if  $\{i_0,i_1\}\subset\{k+1,...,i\}$,
$${\bf c}\pm\frac{m}{n+1}({\bf v}_{i_0}-{\bf v}_{i_1})\in\partial(F);$$
if $i_0\in\{0,1,...,k\}$ and $i_1\in\{k+1,...,i\}$,
$${\bf c}+\frac{m}{n+1}({\bf v}_{i_0}-{\bf v}_{i_1})\in\partial(F)$$
and
$${\bf c}-\frac{m+1}{n+1}({\bf v}_{i_0}-{\bf v}_{i_1})\in\partial(F).$$
In addition, let $Q=\{{\bf c}':{\bf c}'={\bf c}\pm\frac1{n+1}({\bf v}_{i_0}-{\bf v}_{i_1})\text{ for }i_0,i_1\in\{0,1,...,i\},\text{and }i_0\not=i_1\}$. Apparently, $\text{relint}[\text{conv}(Q)]\cap\mathbb{Z}^n=\emptyset$.

Thus, one can get 
$$\lambda_1(F_i|L'^\perp,\mathbb{Z}^n|L'^\perp)=\frac1{m+1}=\frac1{\lfloor\frac{n}{i+1}\rfloor+1}.$$
\end{proof}

\begin{proof}[{\it Proof of Theorem \ref{thm:rhoiTnZn}}]
Let $\mathcal{F}=\{F':F'\text{ is an }i\mathrm{-}\text{dimensional face of }T_n\}$. Given an arbitrary $L\in\mathcal{G}(\mathbb{Z}^n,n-i)$, then $T_n|L^\perp$ is an $i$-dimensional polytope with at least $i+1$ and at most $n+1$ vertices by $V(T_n|L^\perp)=V(T_n)|L^\perp$. As is known,
$$T_n|L^\perp=\underset{F'\in\mathcal{F}}{\bigcup}F'|L^\perp,$$
then there always exists an $i$-dimensional face $F\in\mathcal{F}$ such that ${\bf 0}\in F|L^\perp$ and $\dim(F|L^\perp)=i$.

Firstly, by $F|L^\perp\subseteq T_n|L^\perp$ and Lemma \ref{lem:upTnZn} we have
$$(F|L^\perp)\cap(\mathbb{Z}^n|L^\perp)\subseteq(T_n|L^\perp)\cap(\mathbb{Z}^n|L^\perp)$$
and
$$\lambda_1(T_n|L^\perp,\mathbb{Z}^n|L^\perp)\le\lambda_1(F|L^\perp,\mathbb{Z}^n|L^\perp)\le\frac{1}{\left\lfloor\frac{n}{i+1}\right\rfloor+1}$$ holds for any $L\in\mathcal{G}(\mathbb{Z}^n,n-i)$, where $i=1,2,...,n$. 
Then, it follows by Lemma \ref{prop:projdef} that
$$\rho_i(T_n,\mathbb{Z}^n)\le\frac{1}{\left\lfloor\frac{n}{i+1}\right\rfloor+1}.$$ 

For a lower bound of $\rho_i(T_n,\mathbb{Z}^n)$, by Lemmas \ref{lem:projTnZnFn}-\ref{lem:lowTnZn}, it can be definitely attained by $L'\in\mathcal{G}(\mathbb{Z}^n,n-i)$ constructed in Lemma \ref{lem:projTnZnFn} such that 
$$\lambda_1(T_n|L'^\perp,\mathbb{Z}^n|L'^\perp)=\frac{1}{\left\lfloor\frac{n}{i+1}\right\rfloor+1},$$
yielding
$$\rho_i(T_n,\mathbb{Z}^n)\ge\frac1{\lfloor\frac{n}{i+1}\rfloor+1}.$$


Therefore,
$$\rho_i(T_n,\mathbb{Z}^n)=\frac1{\lfloor\frac{n}{i+1}\rfloor+1}.$$

The proof is completely finished. 
\end{proof}


\bibliographystyle{amsplain}
\bibliography{refs}

\end{document}